\documentclass[11pt,english]{article}
\usepackage[T1]{fontenc}

\newcommand{\ignore}[1]{}%

\usepackage[latin9]{inputenc}
\usepackage{geometry}
\usepackage{calc}
\usepackage{amsthm}
\usepackage{amsmath}
\usepackage{amssymb}
\usepackage{graphicx}

\makeatletter

  \theoremstyle{plain}
  \newtheorem*{theorem*}{\protect\theoremname}
\theoremstyle{plain}
\newtheorem{theorem}{\protect\theoremname}[section]
\theoremstyle{plain}
\newtheorem{conjecture}{\protect\conjecturename}[section]
  \theoremstyle{plain}
  \newtheorem{fact}{\protect\factname}[section]
  \theoremstyle{plain}
  \newtheorem*{proposition*}{\protect\propositionname}
  \theoremstyle{plain}
  \newtheorem{corollary}{\protect\corollaryname}[section]
  \theoremstyle{plain}
  \newtheorem{lemma}{\protect\lemmaname}[section]
  \theoremstyle{remark}
  \newtheorem{claim}{\protect\claimname}[section]
  \theoremstyle{plain}
  \newtheorem*{lemma*}{\protect\lemmaname}
  \theoremstyle{plain}
  \newtheorem*{corollary*}{\protect\corollaryname}
  \theoremstyle{remark}
  \newtheorem*{claim*}{\protect\claimname}
  \theoremstyle{plain}
  \newtheorem*{conjecture*}{\protect\conjecturename}
  \theoremstyle{remark}
  \newtheorem*{acknowledgement*}{\protect\acknowledgementname}
  \theoremstyle{plain}
  \newtheorem*{fact*}{\protect\factname}

\renewcommand{\overline}{\bar }

\providecommand{\E}{\mathrm{E}}
\newcommand{\Var}{\mathrm{Var}}

\makeatother

  \providecommand{\acknowledgementname}{Acknowledgement}
  \providecommand{\claimname}{Claim}
  \providecommand{\conjecturename}{Conjecture}
  \providecommand{\corollaryname}{corollaryollary}
  \providecommand{\factname}{Fact}
  \providecommand{\lemmaname}{lemmama}
  \providecommand{\propositionname}{propositionosition}
  \providecommand{\theoremname}{Theorem}
\providecommand{\theoremname}{Theorem}

\title
{Boolean functions whose Fourier transform is concentrated on pairwise
disjoint subsets of the input}

\author{Aviad Rubinstein\thanks{UC Berkeley. aviad@eecs.berkeley.edu} 
\and Muli Safra\thanks{Tel-Aviv University. muli.safra@gmail.com}}

\begin{document}
\maketitle

\begin{abstract}
We consider Boolean functions $f\colon\left\{ \pm1\right\} ^{m}\rightarrow\left\{ \pm1\right\} $
that are close to a sum of independent functions $\left\{ f_{j}\right\} $
on mutually exclusive subsets of the variables $\left\{ I_{j}\right\} \subseteq P\left(\left[m\right]\right)$.
We prove that any such function is close to just a single function
$f_{k}$ on a single subset $I_{k}$. 

We also consider Boolean functions $f\colon\mathbb{R}^{n}\rightarrow\left\{ \pm1\right\} $
that are close, with respect to any product distribution over $\mathbb{R}^{n}$,
to a sum of their variables. We prove that any such function is close
to one of the variables.

Both our results are independent of the number of variables, but depend
on the variance of $f$. I.e., if $f$ is $\left(\epsilon\cdot{\Var}f\right)$-close
to a sum of independent functions or random variables, then it is
$O\left(\epsilon\right)$-close to one of the independent functions
or random variables, respectively. We prove that this dependence on
${\Var}f$ is tight.

Our results are a generalization of \cite{FKN}, who proved a similar
statement for functions $f\colon\left\{ \pm1\right\} ^{n}\rightarrow\left\{ \pm1\right\} $
that are close to a linear combination of uniformly distributed Boolean
variables.
\end{abstract}



\section{Introduction}

Perhaps the simplest characteristic of functions is linearity%
\footnote{\emph{Linearity} refers to having degree one over $\mathbb{R}^{n}$;
do not confuse with linear functions over $\mathbf{GF}_{2}$, which
are simply parity functions.%
} , i.e. functions which are simply (weighted) sums of their variables.
The set of linear Boolean functions is rather limited: the only linear
Boolean functions are constant functions and dictatorships, i.e. functions
that depend on only one variable. 

Relaxing the notion of linearity, we say that a Boolean function $f\colon\left\{ \pm1\right\} ^{n}\rightarrow\left\{ \pm1\right\} $
is \emph{approximately linear} if it can be approximated by an affine
function of its variables, i.e. $f\approx\sum a_{i}x_{i}+a_{0}$.
Another, equivalent formulation of approximately linear asserts that
$f$'s Fourier coefficients are concentrated on the $1$-st and $0$-th
levels, i.e. $f^{>2}\approx0$. For example, from the latter definition
it is not hard to see that such functions in particular have low \emph{noise
sensitivity}. Informally, low noise sensitivity means that adding
a small random perturbation to the input $x$, is unlikely to change
the value of $f\left(x\right)$.

A theorem of Friedgut, Kalai, Naor proves that those approximately
linear functions have a unique structure: they are approximated by
dictatorships.
\begin{theorem*}
\textup{(FKN Theorem, Informal \cite{FKN})} Every balanced Boolean
function that is almost linear is almost a dictatorship.
\end{theorem*}
Intuitively, one may expect such results to be true, because a linear
combination that is ``well-spread'' among many independent variables
(i.e. far from dictatorship of one variable) should be distributed
similarly to a ``bell-curved'' Gaussian; in particular, it should
be far from the $\pm1$ distribution of a Boolean function which is
\emph{bimodal}, i.e. has two distinct modes or ``peaks'' at $-1$
and $+1$.

\subsection{The long code and related works}

One of the most important historical driving forces in the study of
Boolean functions has been their applications to testing of error
correcting codes \cite{Bou}. In particular, the \emph{long code}
\cite{Long_code} can be viewed as evaluations of dictatorships functions.
Each codeword in the long code corresponds to the evaluation of a
dictatorship $f\left(X\right)=x_{i}$ on all the points on the $n$-dimensional
Boolean hypercube. Indeed, the long code is highly inefficient - since
there are only $n$ possible dictatorships, the long code encodes
$\log n$ bits of information in a $2^{n}$-bit codeword. Despite
its low rate, the long code is an important tool in many results on
hardness of approximation and probabilistically checkable proofs (such
as \cite{Long_code,long_code_app,long_code_app_2,KS_thesis,Fri_app_NP_hardness_VC,Bou_app_UG_hardness_Cuts_embeddability,FKN_app_dinur_PCP,Fri_app_hardness_MultiCut,Fri_app_UG_hardness_VC,Fri_app_UG_hardness_2IS}). 

The great virtue of the long code is that it is a \emph{locally testable
code}: It is possible to distinguish, with high probability, between
a legal codeword and a string that is far from any legal codeword,
by querying just a few random bits of the string. Naturally, this
is a highly desirable property when constructing probabilistically
checkable proofs, which are proofs that must be verified by reading
a few random bits of the proof. Using \emph{local} queries, it is
possible to estimate whether a Boolean function is approximately linear.
These properties can be used by long-code testers \cite{FKN_app_dinur_PCP}
together with the FKN Theorem described above.

\subsection{Our results}

In this work extend the intuition from the FKN Theorem, that a well-spread
sum of independent variables must be far from Boolean. In particular
we ask the following questions:
\begin{enumerate}
\item What happens when the variables are not uniformly distributed over
$\left\{ \pm1\right\} $? In particular, we consider variables which
are not even Boolean or symmetric.

In a social choice setting, it may be intuitive to consider a mechanism
that takes into account how strong is each voter's preference. For
example, in some countries the elections are known to be highly influenced
by the donations the candidates manage to collect (``argentocracy'').

In the context of computational complexity, Boolean analysis theorems
that consider non-uniform distributions have proven very useful. In
particular, Dinur and Safra use the \emph{$p$-biased long code}
in their proof of NP-hardness of approximation of the Vertex Cover
problem \cite{Fri_app_NP_hardness_VC}. In the $p$-biased long code
each codeword corresponds to a dictatorship, in a population where
each voter independently chooses $-1$ with probability $0<p<\frac{1}{2}$
and $+1$ with probability $1-p$. An extension of Friedgut's Junta
Lemma \cite{Fri98} to such non-uniform product distributions was
key to Dinur and Safra's proof.

In this work we prove that even when the variables are not uniformly
distributed over $\left\{ \pm1\right\} $, every Boolean function
that is close to their sum must be close to one of them:
\begin{theorem*}
\textup{(Theorem \ref{sum of sequence} for balanced functions, Informal)}
Every balanced Boolean function that is almost a linear combination
of independent functions (not necessarily Boolean or symmetric) is
almost a dictatorship.
\end{theorem*}
\item What happens when rather than a sum of variables, we have a sum of
functions over disjoint subsets of Boolean variables?

In a social choice setting, it may be intuitive to consider a situation
where the population is divided into tribes; each tribe has an arbitrarily
complex internal mechanism, but the outcomes of all the tribes must
be aggregated into one communal decision by a simple (i.e. almost
linear) mechanism.

Furthermore, this theorem may lead to interesting applications in
computational theory settings where such special structures arise.
In fact, this was our original motivation for this work.

Observe that this question is tightly related to the previous question
because each arbitrary function over a subset of Boolean variables
can be viewed as an arbitrarily-distributed random variable.

In this work we prove that any balanced function that is close to
a sum of functions over disjoint subsets of its variables is almost
completely determined by a function on a single subset:
\begin{theorem*}
\textup{(Corollary \ref{fkn-like} for balanced functions, Informal)}
Every balanced Boolean function that is close to a sum of functions
on mutually exclusive subsets of the variables is close to a dictatorship
by one subset of the variables.
\end{theorem*}
\end{enumerate}
As we will see later, the precise statement of the FKN Theorem does
not require the Boolean function to be balanced. If we do not require
the function to be balanced, there is an obvious exception to the
theorem - \emph{constant functions}, $f\left(X\right)=1$ and $f\left(X\right)=-1$,
are not dictatorships but are considered linear. The general statement
of the FKN Theorem says that a Boolean function that is almost linear
is either almost a dictatorship or almost a constant function. More
precisely, it says that the distance%
\footnote{For the ease of introduction, we use the word ``distance'' in an
intuitive manner throughout this section. However, formally we will
use the squared-$L^{2}$ \emph{semi}distance. See Section \ref{sub:-squared-semi-metric}
for more details. %
} of any Boolean function from the nearest linear (not necessarily
Boolean) function is smaller by at most a constant multiplicative
factor than the distance from either a dictatorship or a constant
function.

One may hope to extend this relaxation to non-Boolean random variables
or subsets of Boolean random variables. E.g. we would like to claim
that the distance of any Boolean function from a sum of functions
on mutually exclusive subsets of the variables is at most the distance
from a function on a single subset or a constant function. However,
it turns out that this is not the case - in Lemma \ref{tightness of main result-1}
we show that this naive extension of the FKN Theorem is false!

The \emph{variance} of a Boolean function measures how far it is
from a constant (either $-1$ or $1$). For example, the variance
of any balanced Boolean function is $1$, whereas any constant function
has a variance of $0$. In order to extend our results to non-balanced
Boolean functions, we have to correct for the low variance. In Theorem
\ref{sum of sequence} and Corollary \ref{fkn-like} we prove that
the above two theorems extend to non-balanced functions relatively
to the variance:
\begin{theorem*}
\textup{(Theorem \ref{sum of sequence}, Informal)} Every Boolean-valued
function that is $(\epsilon\cdot\mbox{variance})$-close to a linear
combination of independent random variables (not necessarily Boolean
or symmetric) is $\epsilon$-close to a dictatorship.
\end{theorem*}

\begin{theorem*}
\textup{(Corollary \ref{fkn-like}, Informal)} Every Boolean function
that is $(\epsilon\cdot\mbox{variance})$-close to a sum of functions
on mutually exclusive subsets of the variables is $\epsilon$-close
to a dictatorship by one subset of the variables.
\end{theorem*}
Intuitively these amendments to our main theorems mean that in order
to prove that a Boolean function is close to a dictatorship, we must
show that it is \emph{very} close to linear.

Finally, in Lemma \ref{tightness of main result-1} we show that this
dependence on the variance is necessary and \emph{tight}.

\subsection{Hypercontractivity}

Many theorems about Boolean functions rely on hypercontractivity theorems
such as the Bonami-Beckner Inequality (\cite{Bonami,Beckner}). Writing
a real-valued function over $\left\{ \pm1\right\} ^{n}$ as a polynomial
yields a distribution over the monomials' degrees $\left\{ 0,\dots n\right\} $,
where the weight of $k$ is the sum of relative weights of monomials
of degree $k$. Hypercontractivity inequalities bound the ratios between
norms of real-valued functions over $\left\{ \pm1\right\} ^{n}$ in
terms of this distribution of weights over their monomials' degrees.
In this work it is not clear how to use such inequalities because
the functions in question may have an arbitrary weight on high degrees
\emph{within} each subset.

All of the proofs presented in this work are completely self-contained
and based on \emph{elementary methods}. In particular, we do not
use any hypercontractivity theorem. This simplicity makes our work
more accessible and intuitive. This trend is exhibited by some recent
related works, e.g. \cite{Sums_of_ind_symmetric_variables,Sums_of_iid_variables,Bou_for_Gaussian_variables},
that also present proofs that do not use hypercontractivity.

\subsection{Organization}

We begin with some preliminaries in Section \ref{sec:Preliminaries}.
In Section \ref{sec:Related-Works} we give a brief survey of related
works. In Section \ref{sec:Our-Main-Results} we formally state our
results. In Section \ref{sec:High-Level-Outline} we give an intuitive
sketch of the proof strategy. The interesting ingredients of the proof
appear in Section \ref{sec:Proofs}, whereas some of the more tedious
case analyses are postponed to Section \ref{sec:Proofs-of-technical}.
Tightness for some of the results is shown in Section \ref{sec:Tightness-of-results}.
Finally, in Section \ref{sec:Conjectures-and-extensions} we make
some concluding comments and discuss possible extensions.

\section{\label{sec:Preliminaries}Preliminaries}

\subsection{$L^{2}$-squared semi-metric\label{sub:-squared-semi-metric}}

Throughout the paper, we define ``closeness'' of random variables
using the squared $L^{2}$-norm:
\[
\left\Vert X-Y\right\Vert _{2}^{2}={\E}\left[\left(X-Y\right)^{2}\right].
\]
It is important to note that this is a \emph{semi}-metric as it
does not satisfy the triangle inequality. Instead, we will use the
$2$-relaxed triangle inequality:
\begin{fact}
\label{fac: relaxed triangle ineq}
\[
\left\Vert X-Y\right\Vert _{2}^{2}+\left\Vert Y-Z\right\Vert _{2}^{2}\geq\frac{1}{2}\left\Vert X-Z\right\Vert _{2}^{2}.
\]
\end{fact}
\begin{proof}
\[
\left\Vert X-Y\right\Vert _{2}^{2}+\left\Vert Y-Z\right\Vert _{2}^{2}\geq\frac{1}{2}\left(\left\Vert X-Y\right\Vert _{2}+\left\Vert Y-Z\right\Vert _{2}\right)^{2}\geq\frac{1}{2}\left\Vert X-Z\right\Vert _{2}^{2}.
\]

\end{proof}
Although it is not a metric, the squared $L^{2}$-norm has some advantages
when analyzing Boolean functions. In particular, when comparing two
Boolean functions, the squared $L^{2}$-norm does satisfy the triangle
inequality because it is simply four times the Hamming distance, and
also twice the $L^{1}$-norm (``Manhattan distance''): $\left\Vert f-g\right\Vert _{2}^{2}=4\cdot\Pr\left[f\neq g\right]=2\left\Vert f-g\right\Vert _{1}$.

Additionally, the squared $L^{2}$-norm behaves ``nicely'' with
respect to the Fourier transform:
\begin{fact}
\label{fac:l2_norm-ft}
\[
\left\Vert f-g\right\Vert _{2}^{2}=\sum\left(\widehat{f}\left(S\right)-\widehat{g}\left(S\right)\right)^{2}.
\]

\end{fact}
(Proofs of Facts \ref{fac:l2_norm-ft}-\ref{fac:var-min_l2_norm}
are standard and are included in the appendix for completeness.)

\subsection{Variance}

The variance of random variable $X$ is defined as 
\[
{\Var}X={\E}\left[X^{2}\right]-\left({\E}X\right)^{2}.
\]
Observe that for a function $f$ the variance can also be defined
in terms of its Fourier coefficients,
\begin{fact}
\label{fac:var-ft}
\[
{\Var}f=\sum_{S\neq\emptyset}\widehat{f}\left(S\right)^{2}.
\]

\end{fact}
Another useful way to define the variance is the expected squared
distance between two random evaluations:
\begin{fact}
\label{fac: (z1-z2)}For any random variable X, \textup{
\[
{\Var}X=\frac{1}{2}\cdot{\E}_{x_{1},x_{2}\sim X\times X}\left(x_{1}-x_{2}\right)^{2}.
\]
}
\end{fact}
We can also view the variance as the $L^{2}$-squared semidistance
from the expectation
\begin{fact}
\label{fac:var-l2_norm}
\[
{\Var}X=\left\Vert X-{\E}X\right\Vert _{2}^{2}.
\]

\end{fact}
Recall also that the expectation ${\E}X$ minimizes this semi-distance
$\left\Vert X-{\E}X\right\Vert _{2}^{2}$: 
\begin{fact}
\label{fac:var-min_l2_norm}
\[
{\Var}X=\min_{E\in\mathbb{R}}\left\Vert X-E\right\Vert _{2}^{2}.
\]

\end{fact}
Finally, for any two functions $f,g$ that are closed in $L^{2}$-squared
semimetric, we can use the $2$-relaxed triangle inequality (Fact
\ref{fac: relaxed triangle ineq}) to bound the difference in variance:
\begin{fact}
\label{fac: L2 norm and var}
\[
{\Var}f\geq\frac{1}{2}{\Var}g-\left\Vert f-g\right\Vert _{2}^{2}.
\]
\end{fact}
\begin{proof}
\[
{\Var}f+\left\Vert f-g\right\Vert _{2}^{2}=\left\Vert f-{\E}f\right\Vert _{2}^{2}+\left\Vert f-g\right\Vert _{2}^{2}\geq\frac{1}{2}\left\Vert g-{\E}f\right\Vert _{2}^{2}\geq\frac{1}{2}\left\Vert g-{\E}g\right\Vert _{2}^{2}.
\]

\end{proof}

\section{Related Work\label{sec:Related-Works}}

In their seminal paper \cite{FKN}, Friedgut, Kalai, and Naor prove
that if a Boolean function is $\epsilon$-close to linear, then it
must be $\left(K\cdot\epsilon\right)$-close to a dictatorship or
a constant function.
\begin{theorem*}
\textup{(FKN Theorem \cite{FKN})} Let $f\colon\left\{ \pm1\right\} ^{n}\rightarrow\left\{ \pm1\right\} $
be a Boolean function, and suppose that $f$'s Fourier transform is
concentrated on the first two levels:
\[
\sum_{\left|S\right|\leq1}\hat{f}{}^{2}\left(S\right)\geq1-\epsilon.
\]
Then for some universal constant $K$: 
\begin{enumerate}
\item either $f$ is $\left(K\cdot\epsilon\right)$-close to a constant
function; i.e. for some $\sigma\in\left\{ \pm1\right\} $
\[
\left\Vert f-\sigma\right\Vert _{2}^{2}\leq K\cdot\epsilon;
\]

\item or $f$ is $\left(K\cdot\epsilon\right)$-close to a dictatorship;
i.e. there exists $k\in\left[n\right]$ and $\sigma\in\left\{ \pm1\right\} $
such that  $f$: 
\[
\left\Vert f-\sigma\cdot x_{k}\right\Vert _{2}^{2}\leq K\cdot\epsilon.
\]
$\square$
\end{enumerate}
\end{theorem*}
The FKN Theorem quickly found applications in social choice theory
\cite{Kalai_social}. More importantly, it has since been applied
in other fields; a good example is Dinur's combinatorial proof of
the PCP theorem \cite{FKN_app_dinur_PCP}.

There are also many works on generalizations on the FKN Theorem. Alon
et al. \cite{FKN_for_Z_r_1} and Ghandehari and Hatami \cite{FKN_for_Z_r_2}
prove generalizations for functions with domain $\mathbb{Z}_{r}^{n}$
for $r\geq2$. Friedgut \cite{FKN_for_any_deg_and_non-uniform_measure}
proves a similar theorem that also holds for Boolean functions of
higher degrees and over non-uniform distributions; however, this theorem
requires bounds on the expectation of the Boolean function. In \cite{FKN_quantum},
Montanaro and Osborne prove quantum variants of the FKN Theorem for
any ``quantum Boolean functions'', i.e. any unitary operator $f$
such that $f^{2}$ is the identity operator. Falik and Friedgut \cite{FKN_for_representation_theory}
and Ellis et al. \cite{EFF-FKN_balanced_Sn,EFF-FKN_Sn} prove representation-theory
variants of the FKN Theorem, for functions which are close to a linear
combination of an irreducible representation of elements of the symmetric
group.

The FKN Theorem is an easy corollary once the following proposition
is proven: If the \emph{absolute value} of the linear combination
of Boolean variables has a small variance, then it must be concentrated
on a single variable. Formally,
\begin{proposition*}
\textup{(FKN Proposition \cite{FKN})} Let $\left(X_{i}\right)_{i=1}^{n}$
be a sequence of independent symmetric variables with supports $\left\{ \pm a_{i}\right\} $
such that $\sum a_{i}^{2}=1$. For some universal constant $K$, if
\[
{\Var}\left|\sum_{i}X_{i}\right|\leq\epsilon,
\]
 then for some $k\in\left[n\right]$
\[
a_{k}>1-K\cdot\epsilon.
\]
$\square$
\end{proposition*}
Intuitively, this proposition says that if the variance was spread
among many of the variables, i.e. the ``weights'' $\left(a_{i}\right)$
were somewhat evenly distributed, then one would expect the sum of
such independent variables to be closer to a Gaussian rather than
a bimodal distribution around $\pm1$.

This proposition has been generalized in several ways in a sequence
of recent works by Wojtaszczyk \cite{Sums_of_ind_symmetric_variables}
and Jendrej, Oleszkiewicz, and Wojtaszczyk \cite{Sums_of_iid_variables},
which are of particular interest to us. Jendrej et al. prove extensions
of the FKN Proposition to the following cases:
\begin{enumerate}
\item The case where $X_{i}$'s are independent symmetric

\begin{theorem*}
\textup{(\cite{Sums_of_iid_variables})} Let $\left(X_{i}\right)_{i=1}^{n}$
be a sequence of independent \textbf{symmetric} variables. Then there
exists an universal constant $K$, such that for some $k\in\left[n\right]$
\[
\inf_{E\in\mathbb{R}}{\Var}\left|\sum_{i}X_{i}+E\right|\geq\frac{{\Var}\sum_{i\neq k}X_{i}}{K}.
\]
$\square$
\end{theorem*}
\item The case where all the $X_{i}$'s are identically distributed

\begin{theorem*}
\textup{(\cite{Sums_of_iid_variables})} Let $\left(X_{i}\right)_{i=1}^{n}$
be a sequence of \textbf{i.i.d.} variables which are not constant
a.s.. Then there exists a $K_{X}$, which depends only on the distribution
from which the $X_{i}$'s are drawn, such that for any sequence of
real numbers $\left(a_{i}\right)_{i=1}^{n}$, for some $k\in\left[n\right]$
\[
\inf_{E\in\mathbb{R}}{\Var}\left|E+\sum_{i}a_{i}X_{i}\right|\geq\frac{\sum_{i\neq k}a_{i}^{2}}{K_{X}}.
\]
$\square$
\end{theorem*}
\end{enumerate}

\subsubsection*{Concurrent Progress by Jendrej, Oleszkiewicz, and Wojtaszczyk}

Let us note that the works of Wojtaszczyk \cite{Sums_of_ind_symmetric_variables}
and Jendrej, Oleszkiewicz, and Wojtaszczyk \cite{Sums_of_iid_variables}
have been eventually extended and transformed into \cite{Krzysztof_article},
which is conditionally accepted \textit{Theory of Computing}.

The extension \cite{Krzysztof_article} --carried out independently
of our work-- has resulted in a theorem which is our Theorem \ref{sum of sequence},
however for the case of bounded-variance variables.

Furthermore, it is worthwhile noting that the proof there could be
amended so as to achieve Theorem \ref{sum of sequence}, and can thus
be considered as an alternative technique for such purposes.

Following the work of Jendrej et al. and the announcement of our results,
Nayar \cite{Nayar13_biased-FKN} proved a variant of the FKN Theorem
for the biased hypercube, which builds on ideas from \cite{Krzysztof_article}.

\section{Formal Statement of Our Results\label{sec:Our-Main-Results}}

In this work we consider the following relaxation of linearity in
the premise of the FKN Theorem: Given a partition of the variables
$\left\{ I_{j}\right\} $ and a function $f_{j}$ (not necessarily
Boolean or symmetric) on the variables in each subset, we look at
the premise that the Boolean function $f$ is close to a linear combination
of the $f_{j}$'s. Our main result (Corollary \ref{fkn-like}) states,
loosely, that any such $f$ must be close to being completely dictated
by its restriction $f_{k}$ to the variables in a single subset $I_{k}$
of the partition. 

While making a natural generalization of the well-known FKN Theorem,
our work also has a surprising side: In the FKN Theorem and similar
results, if a function is $\epsilon$-close to linear then it is $\left(K\cdot\epsilon\right)$-close
to a dictatorship, for some constant $K$. We prove that while this
is true in the partition case for balanced functions, it does not
hold in general. In particular, we require $f$ to be $\left(\epsilon\cdot{\Var}f\right)$-close
to linear in the $f_{j}$'s, in order to prove that it is only $\left(K\cdot\epsilon\right)$-close
to being dictated by some $f_{k}$. We show (Lemma \ref{tightness of main result-1})
that this dependence on ${\Var}f$ is tight.

Our first result is a somewhat technical theorem, generalizing the
FKN Proposition. We consider the sum $\sum_{i=1}^{n}X_{i}$ of a sequence
of independent random variables. In particular, we do not assume that
the variables are Boolean, symmetric, balanced, or identically distributed.
Our main technical theorem, which generalizes the FKN Proposition,
states that if this sum does not ``behave like'' any single variable
$X_{k}$, then it is also far from Boolean. In other words, if a sum
of independent random variables is close to a Boolean function then
most of its variance comes from only one variable.

We show that $\sum X_{i}$ is far from Boolean by proving a lower
bound on the variance of its absolute value, ${\Var}\left|\sum X_{i}\right|$.
Note that for any Boolean function $f$, $\left|f\right|=1$ everywhere,
and thus ${\Var}\left|f\right|=0$. Therefore the lower bound on ${\Var}\left|\sum X_{i}\right|$
is in fact also a lower bound on the (semi-)distance from the nearest
Boolean function:
\[
{\Var}\left|\sum X_{i}\right|\leq\min_{\mbox{\ensuremath{f}\,\ is Boolean}}\left\Vert f-\sum X_{i}\right\Vert _{2}^{2}.
\]

By saying that the sum $\sum_{i=1}^{n}X_{i}$ ``behaves like'' a
single variable $X_{k}$, we mean that their difference is almost
a constant function; i.e. that 
\[
\min_{c\in\mathbb{R}}\left\Vert X_{k}-\sum_{i}X_{i}-c\right\Vert _{2}^{2}=\left\Vert \sum_{i\neq k}X_{i}-{\E}\sum_{i\neq k}X_{i}\right\Vert _{2}^{2}={\Var}\sum_{i\neq k}X_{i}
\]

is small.

Furthermore, the definition of ``small'' depends on the expectation
and variance of the sum of the sequence, which we denote by $E$ and
$V$ 
\begin{eqnarray*}
E & = & {\E}\sum X_{i}=\sum{\E}X_{i}\\
V & = & {\Var}\sum X_{i}=\sum{\Var}X_{i}.
\end{eqnarray*}
Formally, our main technical theorem states that
\begin{theorem}
\label{sum of sequence}Let $\left(X_{i}\right)_{i=1}^{n}$ be a sequence
of independent (not necessarily symmetric) random variables, and let
$E$ and $V$ be the expectation and variance of their sum, respectively.
Then for some universal constant $K_{2}\leq61440$ we have that there
exists $k\in\left[n\right]$ such that,
\[
{\Var}\left|\sum_{i}X_{i}\right|\geq\frac{V\cdot{\Var}\sum_{i\neq k}X_{i}}{K_{2}\left(V+E^{2}\right)}.
\]
$\square$
\end{theorem}
The main motivation for proving this theorem is that it implies a
generalization of the FKN Theorem, Corollary \ref{fkn-like} below. 

Intuitively, while the FKN Theorem holds for Boolean functions that
are almost linear in \emph{individual} variables, we generalize
it to functions that are almost linear with respect to a \emph{partition}
of the variables. 

Formally, let $f\colon\left\{ \pm1\right\} ^{m}\rightarrow\left\{ \pm1\right\} $
and let $I_{1},\dots,I_{n}$ be a partition of $\left[m\right]$;
denote by $f_{j}$ the restriction of $f$ to each subset of variables:
\[
f_{j}=\sum_{\emptyset\neq S\subseteq I_{j}}\widehat{f}\left(S\right)\chi_{S}.
\]
Our main corollary states that if $f$ behaves like the sum of the
$f_{j}$'s then it behaves like some single $f_{k}$:
\begin{corollary}
\label{fkn-like}Let $f$, $I_{j}$'s, and $f_{j}$'s be as defined
above. Suppose that $f$ is concentrated on coefficients that do not
cross the partition, i.e.:
\[
\sum_{S\colon\exists j,\, S\subseteq I_{j}}\hat{f}{}^{2}\left(S\right)\geq1-\left(\epsilon\cdot{\Var}f\right).
\]
Then for some $k\in\left[n\right]$, $f$ is close to $f_{k}+\widehat{f}\left(\emptyset\right)$:
\[
\left\Vert f-f_{k}-\widehat{f}\left(\emptyset\right)\right\Vert _{2}^{2}\leq\left(K_{2}+2\right)\cdot\epsilon.
\]
$\square$
\end{corollary}
In particular, notice that it implies that $f$ is concentrated on
the variables in a single subset $I_{k}$. 

Unlike the FKN Theorem and many similar statements, it does not suffice
to assume that $f$ is $\epsilon$-close to linear. Our main results
require a dependence on the variance of $f$. We prove in Section
\ref{sub:Tightness-of-the-main} that this dependence is tight up
to a constant factor by constructing an example for which ${\Var}f=o\left(1\right)$
and $f$ is $\left(\epsilon\cdot{\Var}f\right)$-close to linear with
respect to a partition, but $f$ is still $\Omega\left(\epsilon\right)$-far
from being dictated by any subset.
\begin{lemma}
\label{tightness of main result-1}Corollary \ref{fkn-like} is tight
up to a constant factor. In particular, the factor ${\Var}f$ is necessary.

More precisely, there exists a series of functions $f^{\left(m\right)}\colon\left\{ \pm1\right\} ^{2m}\rightarrow\left\{ \pm1\right\} $
and partitions $\left(I_{1}^{\left(m\right)},I_{2}^{\left(m\right)}\right)$
such that the restrictions $\left(f_{1}^{\left(m\right)},f_{2}^{\left(m\right)}\right)$
of $f^{\left(m\right)}$ to variables in $I_{j}^{\left(m\right)}$
satisfy 
\[
\sum_{S\colon\exists j,\, S\subseteq I_{j}^{\left(m\right)}}\hat{f}{}^{2}\left(S\right)=1-O\left(2^{-m}\cdot{\Var}f\right),
\]
but for every $j\in\left\{ 1,2\right\} $
\[
\left\Vert f^{\left(m\right)}-f_{j}^{\left(m\right)}-\widehat{f^{\left(m\right)}}\left(\emptyset\right)\right\Vert _{2}^{2}=\Theta\left(2^{-m}\right)=\omega\left(2^{-m}\cdot{\Var}f\right).
\]

\end{lemma}

\subsection{From our results to the FKN Theorem}

We claim that our results generalize the FKN Theorem. For a constant
variance, the FKN Theorem indeed follows immediately from Corollary
\ref{fkn-like} (for some worse constant $K_{\mbox{FKN}}\geq\frac{K_{2}}{{\Var}f}$).
However, because the premise of Corollary \ref{fkn-like} depends
on the variance, it may not be obvious how to obtain the FKN Theorem
for the general case, where the variance may go to zero. Nonetheless,
we note that thanks to an observation by Guy Kindler \cite{FKN} the
FKN Theorem follows easily once the special case of balanced functions
is proven.

Given a Boolean function $f\colon\left\{ \pm1\right\} ^{n}\rightarrow\left\{ \pm1\right\} $,
we define a \emph{balanced} Boolean function $g\colon\left\{ \pm1\right\} ^{n+1}\rightarrow\left\{ \pm1\right\} $
that will be \emph{as close to linear} as $f$, 
\[
g\left(x_{1},x_{2},\dots,x_{n},x_{n+1}\right)=x_{n+1}\cdot f\left(x_{n+1}\cdot x_{1},x_{n+1}\cdot x_{2},\dots,x_{n+1}\cdot x_{n}\right).
\]
First, notice that $g$ is indeed balanced because:
\begin{gather*}
2{\E}\left[g\left(X;x_{n+1}\right)\right]={\E}\left[f\left(X\right)\right]-{\E}\left[f\left(-X\right)\right]={\E}\left[f\left(X\right)\right]-{\E}\left[f\left(X\right)\right]=0,
\end{gather*}
where the second equality holds because under uniform distribution
taking the expectation over $X$ is the same as taking the expectation
over $-X$. 

Observe also that every monomial $\widehat{f}\left(S\right)\chi_{S}\left(X\right)$
in the Fourier representation of $f\left(X\right)$ is multiplied
by $x_{n+1}^{\left|S\right|+1}$ in the Fourier transform of $g\left(X;x_{n+1}\right)$.
(The $\left|S\right|+1$ in the exponent comes from $\left|S\right|$
for all the variables that appear in the monomial, and another $1$
for the $x_{n+1}$ outside the function). Since $x_{n+1}\in\left\{ \pm1\right\} $,
for odd $\left|S\right|$ we have that $x_{n+1}^{\left|S\right|+1}=1$,
and the monomial does not change, i.e. $\widehat{g}\left(S\right)=\widehat{f}\left(S\right)$;
for even $\left|S\right|$, $x_{n+1}^{\left|S\right|+1}=x_{n+1}$,
so $\widehat{g}\left(S\cup\left\{ n+1\right\} \right)=\widehat{f}\left(S\right)$.
In particular, the total weight on the first and zeroth level of the
Fourier representation is preserved because 
\begin{gather}
\forall i\in\left[n\right]\;\widehat{g}\left(\left\{ i\right\} \right)=\widehat{f}\left(\left\{ i\right\} \right);\,\,\,\widehat{g}\left(\left\{ n+1\right\} \right)=\widehat{f}\left(\emptyset\right).\label{eq:guys_trick}
\end{gather}

If $f$ satisfies the premise for the FKN Theorem, i.e. if $\sum_{\left|S\right|\leq1}\hat{f}{}^{2}\left(S\right)\geq1-\epsilon$,
then from \eqref{eq:guys_trick} it is clear that the same also holds
for $g$. From the FKN Theorem for the balanced special case we deduce
that $g$ is $\left(K\cdot\epsilon\right)$-close to a dictatorship,
i.e. there exists $k\in\left[n+1\right]$ such that $\widehat{g}\left(\left\{ k\right\} \right)^{2}\geq1-\left(K\cdot\epsilon\right)$.
Therefore by \eqref{eq:guys_trick} $f$ is also $\left(K\cdot\epsilon\right)$-close
to either a dictatorship (when $k\in\left[n\right]$) or a constant
function (when $k=n+1$). The FKN Theorem for balanced functions follows
as a special case of our main results, and therefore this work also
provides an alternative proof for the FKN Theorem.

\section{\label{sec:High-Level-Outline}High-Level Outline of the Proof}

The main step to proving Theorem \ref{sum of sequence} for a sequence
of $n$ variables, is Lemma \ref{lemma: Var|X+Y+c| > K * VarX * VarY}
below which handles the special case of only two random variables.
The main theorem then follows by partitioning the $n$ variables into
two subsets, and labeling their sums $X$ and $Y$, respectively (Subsection
\ref{sub:main-thm}).
\begin{lemma}
\label{lemma: Var|X+Y+c| > K * VarX * VarY}Let $X,Y$ be any two independent
random variables, and let $E$ and $V$ be the expectation and variance
of their sum, respectively. Then for some universal constant $K_{1}\leq20480$,
\begin{equation}
{\Var}\left|X+Y\right|\geq\frac{V\cdot\min\left\{ {\Var}X,{\Var}Y\right\} }{K_{1}\left(V+E^{2}\right)}.\label{eq: lemma statement}
\end{equation}
$\square$
\end{lemma}
Intuitively, in the expression on the left-hand side of \eqref{eq: lemma statement}
we consider the \emph{sum} of two independent variables, which we
may expect to variate more than each variable separately. Per contra,
the same side of \eqref{eq: lemma statement} also has the variance
of the \emph{absolute value}, which is in general smaller than just
the variance (without absolute value). Lemma \ref{lemma: Var|X+Y+c| > K * VarX * VarY}
bounds this loss of variance.

On a high level, the main idea of the proof of Lemma \ref{lemma: Var|X+Y+c| > K * VarX * VarY}
is separation to two cases, depending on the \emph{variance of the
absolute value} of the random variables, relative to the original
variance of the variables (without absolute value):
\begin{enumerate}
\item If both $\left|X+{\E}Y\right|$ and $\left|Y+{\E}X\right|$ have relatively
small variance, then $X+{\E}Y$ and $Y+{\E}X$ can be both approximated
by random variables with \emph{constant absolute values}. In this
case we prove the result by case analysis.
\item If either $\left|X+{\E}Y\right|$ or $\left|Y+{\E}X\right|$ has a
relatively large variance, we prove an auxiliary lemma that states
that the variance of the absolute value of the sum, ${\Var}\left|X+Y\right|$,
is not much smaller than the variance of the \emph{absolute value}
of either variable (${\Var}\left|X+{\E}Y\right|$, ${\Var}\left|Y+{\E}X\right|$):\end{enumerate}
\begin{lemma}
\label{lemma: |X+Y| > K*|X+E=00005BY=00005D|}Let $X,Y$ be any two
independent random variables, and let $E$ be the expectation of their
sum. Then for some universal constant $K_{0}\leq4$,
\[
{\Var}\left|X+Y\right|\geq\frac{\max\left\{ {\Var}\left|X+{\E}Y\right|,{\Var}\left|Y+{\E}X\right|\right\} }{K_{0}}.
\]
$\square$
\end{lemma}
Note that in this lemma, unlike the former statements discussed so
far, the terms on the right-hand side also appear in \emph{absolute
value}. In particular, this makes the inequality hold with respect
to the \emph{maximum} of the two variances.

We find it of separate interest to note that this lemma is tight in
the sense that it is necessary to take a non-trivial constant $K_{0}>1$:
\begin{claim}
\label{cla: tightness of aux lemma}A non-trivial constant is necessary
for Lemma \ref{lemma: |X+Y| > K*|X+E=00005BY=00005D|}. More precisely,
there exist two independent \emph{balanced} random variables $\overline{X},\overline{Y}$,
such that the following inequality does not hold for any value $K_{0}<4/3$:
\[
{\Var}\left|\overline{X}+\overline{Y}\right|\geq\frac{\max\left\{ {\Var}\left|\overline{X}\right|,{\Var}\left|\overline{Y}\right|\right\} }{K_{0}}.
\]

In particular, it is interesting to note that $K_{0}>1$.
\end{claim}
Discussion and proof appear in Section \ref{sub:Tightness-of-the-aux lemma}.

$\square$

\section{\label{sec:Proofs}Proofs}

\subsection{From variance of absolute value to variance of absolute value of
sum: proof of Lemma \ref{lemma: |X+Y| > K*|X+E=00005BY=00005D|} }

We begin with the proof of a slightly more general form of Lemma \ref{lemma: |X+Y| > K*|X+E=00005BY=00005D|}:
\begin{lemma}
\label{lemma: var|X+Y+E| > var|X+E|}Let $\overline{X},\overline{Y}$
be any two independent \emph{balanced} random variables, and let
$E$ be any real number. Then for some universal constant $K_{0}\leq4$,
\[
{\Var}\left|\overline{X}+\overline{Y}+E\right|\geq\frac{\max\left\{ {\Var}\left|\overline{X}+E\right|,{\Var}\left|\overline{Y}+E\right|\right\} }{K_{0}}.
\]

\end{lemma}
Lemma \ref{lemma: |X+Y| > K*|X+E=00005BY=00005D|} follows easily by
taking $E={\E}\left[X+Y\right]$ and $\overline{X}=X-\E X$ and $\overline{Y}=Y-\E Y$.
\begin{proof}
This lemma is relatively easy to prove partly because the right-hand
side contains the maximum of the two variances. Thus, it suffices
to prove separately that the left-hand side is greater or equal to
$\frac{{\Var}\left|\overline{X}+E\right|}{K_{0}}$ and to $\frac{{\Var}\left|\overline{Y}+E\right|}{K_{0}}$.
Without loss of generality we will prove:
\begin{equation}
{\Var}\left|\overline{X}+\overline{Y}+E\right|\geq\frac{{\Var}\left|\overline{X}+E\right|}{K_{0}}.\label{eq: var|X+Y+E| > var|X+E|}
\end{equation}
Separating into two inequalities is particularly helpful, because
now $\overline{Y}$ no longer appears in the right-hand side. 

Our next step is to reduce to the special case where $\overline{Y}$
is a balanced variable with only two values in its support. Every
balanced variable can be written as a mixture of balanced random variables
$\left\{ \overline{Y_{\alpha}}\right\} $, each with support at most
two; this follows by applying the Krein-Milman theorem to the space
of balanced random variables. Now use the convexity of the variance
to get:
\begin{eqnarray*}
{\Var}\left|\overline{X}+\overline{Y}+E\right| & = & {\E}\left(\left|\overline{X}+\overline{Y}+E\right|-{\E}\left|\overline{X}+\overline{Y}+E\right|\right)^{2}\\
 & = & {\E}_{\alpha}\left[{\E}\left(\left|\overline{X}+\overline{Y_{\alpha}}+E\right|-{\E}\left|\overline{X}+\overline{Y}+E\right|\right)^{2}\right]\\
 & \geq & {\E}_{\alpha}\left[{\E}\left(\left|\overline{X}+\overline{Y_{\alpha}}+E\right|-{\E}\left|\overline{X}+\overline{Y_{\alpha}}+E\right|\right)^{2}\right]\\
 & = & {\E}_{\alpha}\left[{\Var}\left|\overline{X}+\overline{Y_{\alpha}}+E\right|\right].
\end{eqnarray*}

Thus ${\Var}\left|\overline{X}+\overline{Y}+E\right|$ is in particular
greater or equal to ${\Var}\left|\overline{X}+\overline{Y}_{\alpha}+E\right|$
for some $\alpha$. Therefore in order to prove Lemma \ref{lemma: var|X+Y+E| > var|X+E|},
it suffices to prove the lower bound ${\Var}\left|\overline{X}+\overline{Y_{\alpha}}+E\right|$
(with respect to ${\Var}\left|\overline{X}+E\right|$) for every balanced
$\overline{Y_{\alpha}}$ with only two possible values. 

Recall (Fact \ref{fac: (z1-z2)}) that we can express the variances
of $\left|\overline{X}+\overline{Y_{\alpha}}+E\right|$ and $\left|\overline{X}+E\right|$
in terms of the expected squared distance between two random evaluations.
We use a simple case analysis to prove that adding any balanced $\overline{Y_{\alpha}}$
with support of size two preserves (up to a factor of $\frac{1}{4}$)
the expected squared distance between any two possible evaluations
of $\left|\overline{X}+E\right|$. 
\begin{claim}
\label{cla:For-every x,x,y,y}For every two possible evaluations $x_{1},x_{2}$
in the support of $\left(\overline{X}+E\right)$,
\[
{\E}_{\left(y_{1},y_{2}\right)\sim\overline{Y_{\alpha}}\times\overline{Y_{\alpha}}}\left(\left|x_{1}+y_{1}\right|-\left|x_{2}+y_{2}\right|\right)^{2}\geq\frac{1}{4}\left(\left|x_{1}\right|-\left|x_{2}\right|\right)^{2}.
\]

\end{claim}

The proof appears in Section \ref{cla:For-every x,x,y,y}.

$\square$

Finally, in order to achieve the bound on the variances (inequality
\eqref{eq: var|X+Y+E| > var|X+E|}), take the expectation over all
choices of $\left(x_{1},x_{2}\right)\sim\left(\overline{X}+E\right)\times\left(\overline{X}+E\right)$:
\begin{eqnarray*}
{\Var}\left|\overline{X}+\overline{Y_{\alpha}}+E\right| & = & \frac{1}{2}{\E}_{x_{1},x_{2},y_{1},y_{2}}\left(\left|x_{1}+y_{1}\right|-\left|x_{2}+y_{2}\right|\right)^{2}\\
 & \geq & \frac{1}{8}{\E}_{x_{1},x_{2}}\left(\left|x_{1}\right|-\left|x_{2}\right|\right)^{2}\\
 & = & \frac{1}{4}\cdot{\Var}\left|\overline{X}+E\right|.
\end{eqnarray*}
(Where the two equalities follow by Fact \ref{fac: (z1-z2)}, and
the inequality by Claim \ref{cla:For-every x,x,y,y}.)
\end{proof}

\subsection{\label{sec: proof Var|X+Y+c| > K * VarX * VarY}From variance of
absolute value of sum to variance: proof of Lemma \ref{lemma: Var|X+Y+c| > K * VarX * VarY}}

We advance to the more interesting Lemma \ref{lemma: Var|X+Y+c| > K * VarX * VarY},
where we bound the variance of $\left|X+Y\right|$ with respect to
the minimum of ${\Var}X$ and ${\Var}Y$. Intuitively, in the expression
on the left-hand side of \eqref{eq:var|x+y| > k varx, vary} we consider
the \emph{sum} of two independent variables, which we may expect
to variate more than each variable separately. Per contra, the same
side of \eqref{eq:var|x+y| > k varx, vary} also has the variance
of the \emph{absolute value}, which is in general smaller than just
the variance (without absolute value). We will now bound this loss
of variance.
\begin{lemma*}
\textup{(Lemma \ref{lemma: Var|X+Y+c| > K * VarX * VarY})} Let $X,Y$
be any two independent random variables, and let $E$ and $V$ be
the expectation and variance of their sum, respectively. Then for
some universal constant $K_{1}\leq20480$,
\begin{equation}
{\Var}\left|X+Y\right|\geq\frac{V\cdot\min\left\{ {\Var}X,{\Var}Y\right\} }{K_{1}\left(V+E^{2}\right)}.\label{eq:var|x+y| > k varx, vary}
\end{equation}
\end{lemma*}
\begin{proof}
We change variables by subtracting the expectation of $X$ and $Y$,
\begin{eqnarray*}
\overline{X} & = & X-{\E}X\\
\overline{Y} & = & Y-{\E}Y.
\end{eqnarray*}
Note that the new variables $\overline{X},\overline{Y}$ are \emph{balanced}.
Also observe that we are now interested in showing a lower bound for
\[
{\Var}\left|X+Y\right|={\Var}\left|\overline{X}+\overline{Y}+E\right|.
\]

On a high level, the main idea of the proof is separation to two cases,
depending on the \emph{variance of the absolute value} of the random
variables, relative to the original variance of the variables (without
absolute value):
\begin{enumerate}
\item If either $\left|\overline{X}+E\right|$ or $\left|\overline{Y}+E\right|$
has a relatively large variance, we can simply apply Lemma \ref{lemma: var|X+Y+E| > var|X+E|}
that states that the variance of the absolute value of the sum, ${\Var}\left|\overline{X}+\overline{Y}+E\right|$,
is not much smaller than the variance of the \emph{absolute value}
of either variable (${\Var}\left|\overline{X}+E\right|$ and ${\Var}\left|\overline{Y}+E\right|$).
\item If both $\left|\overline{X}+E\right|$ and $\left|\overline{Y}+E\right|$
have relatively small variance, then $\overline{X}+E$ and $\overline{Y}+E$
can be both approximated by random variables with \emph{constant
absolute values}, (i.e. random variables with supports $\left\{ \pm d_{X}\right\} $
and $\left\{ \pm d_{Y}\right\} $ for some reals $d_{X}$ and $d_{Y}$,
respectively). For this case we prove the result by case analysis.
\end{enumerate}
Formally, let $0<a<1/10$ be some parameter to be determined later,
and denote 
\begin{equation}
M_{XY}=\min\left\{ {\Var}\overline{X},{\Var}\overline{Y}\right\} =\min\left\{ {\Var}X,{\Var}Y\right\} .\label{eq:M_XY-def}
\end{equation}

\begin{enumerate}
\item If either of the variances of the absolute values is large, i.e.
\begin{eqnarray*}
\max\left\{ {\Var}\left|\overline{X}+E\right|,{\Var}\left|\overline{Y}+E\right|\right\}  & \geq & a\cdot M_{XY},
\end{eqnarray*}
then we can simply apply Lemma \ref{lemma: var|X+Y+E| > var|X+E|} to
obtain: 
\[
{\Var}\left|\overline{X}+\overline{Y}+E\right|\geq\frac{\max\left\{ {\Var}\left|\overline{X}+E\right|,{\Var}\left|\overline{Y}+E\right|\right\} }{K_{0}}\geq\frac{a\cdot M_{XY}}{K_{0}}.
\]

\item On the other hand, if both the variances of the absolute values are
small, i.e.
\begin{eqnarray}
\max\left\{ {\Var}\left|\overline{X}+E\right|,{\Var}\left|\overline{Y}+E\right|\right\}  & < & a\cdot M_{XY},\label{var|X+E| < var X}
\end{eqnarray}
then $\overline{X}+E$ and $\overline{Y}+E$ are almost constant \emph{in
absolute value}.

In particular, let the variables $X'$ and $Y'$ be the constant-absolute-value
approximations to $\overline{X}+E$ and $\overline{Y}+E$, respectively:
\begin{eqnarray*}
X' & = & \mbox{sign}\left(\overline{X}+E\right)\cdot{\E}\left|\overline{X}+E\right|\\
Y' & = & \mbox{sign}\left(\overline{Y}+E\right)\cdot{\E}\left|\overline{Y}+E\right|.
\end{eqnarray*}
From the precondition \eqref{var|X+E| < var X} it follows that $\left(\overline{X}+E\right)$
and $\left(\overline{Y}+E\right)$ are close to $X'$ and $Y'$, respectively:
\begin{eqnarray*}
\left\Vert \overline{X}-\left(X'-E\right)\right\Vert _{2}^{2} & < & a\cdot M_{XY}\\
\left\Vert \overline{Y}-\left(Y'-E\right)\right\Vert _{2}^{2} & < & a\cdot M_{XY}.
\end{eqnarray*}
In particular, by the $2$-relaxed triangle inequality (Facts \ref{fac: relaxed triangle ineq}
and \ref{fac: L2 norm and var}) we have that the following variances
are close:
\begin{eqnarray}
{\Var}\left(X'-E\right) & > & \frac{1}{2}{\Var}\overline{X}-a\cdot{\Var}\overline{X}\label{eq:var x is close}\\
{\Var}\left(Y'-E\right) & > & \frac{1}{2}{\Var}\overline{Y}-a\cdot{\Var}\overline{Y}\label{eq:var y is close}\\
{\Var}\left|\overline{X}+\overline{Y}+E\right| & \geq & \frac{1}{2}{\Var}\left|Y'+X'-E\right|-\left\Vert \left|\overline{X}+\overline{Y}+E\right|-\left|Y'+X'-E\right|\right\Vert _{2}^{2}\nonumber \\
 & > & \frac{1}{2}{\Var}\left|Y'+X'-E\right|-4a\cdot M_{XY}.\label{eq:var x+y is close}
\end{eqnarray}

Hence, it will be useful to obtain a bound equivalent to \eqref{eq:var|x+y| > k varx, vary},
but in terms of the approximating variables, $X',Y'$. We will then
use the similarity of the variances to extend the bound to $\overline{X},\overline{Y}$
and complete the proof of the lemma.

We use case analysis over the possible evaluations of $X'$ and $Y'$
to prove the following claim:
\begin{claim}
\label{cla: const  | |}Let $\overline{X},\overline{Y}$ be balanced
random variables and let $X',Y'$ be the constant-absolute-value approximations
of $\overline{X},\overline{Y}$, respectively:
\begin{eqnarray*}
X' & = & \mbox{sign}\left(\overline{X}+E\right)\cdot{\E}\left|\overline{X}+E\right|\\
Y' & = & \mbox{sign}\left(\overline{Y}+E\right)\cdot{\E}\left|\overline{Y}+E\right|.
\end{eqnarray*}

Then the variance of the \emph{absolute value} of $X'+Y'-E$ is
bounded by:
\[
{\Var}\left|X'+Y'-E\right|\geq\frac{{\Var}\left(X'-E\right){\Var}\left(Y'-E\right)}{16\left({\Var}\overline{X}+E^{2}\right)}.
\]

\end{claim}

The proof appears in Section \ref{sub:Proof-of-analisys for const | |}.$\square$

Now we use the closeness of the approximating variables $X',Y'$ to
recover a bound for the balanced variables $\overline{X},\overline{Y}$:
\begin{eqnarray*}
{\Var}\left|\overline{X}+\overline{Y}+E\right| & \geq & \frac{1}{2}{\Var}\left|X'+Y'-E\right|-4a\cdot M_{XY}\\
 & \geq & \frac{{\Var}\left(X'-E\right){\Var}\left(Y'-E\right)}{32\left(V+E^{2}\right)}-4a\cdot M_{XY}\\
 & \geq & \frac{\left(1-2a\right)^{2}}{128}\frac{{\Var}\overline{X}\cdot{\Var}\overline{Y}}{V+E^{2}}-4a\cdot M_{XY}\\
 & \geq & M_{XY}\cdot\left[\frac{\left(1-2a\right)^{2}}{128}\frac{\max\left\{ {\Var}\overline{X},{\Var}\overline{Y}\right\} }{V+E^{2}}-4a\right]\\
 & \geq & M_{XY}\cdot\left[\frac{1-4a}{256}\frac{V}{V+E^{2}}-4a\right]\\
 & \geq & M_{XY}\cdot\left[\frac{1}{256}\frac{V}{V+E^{2}}-5a\right].
\end{eqnarray*}

(Where the first line follows by equation \eqref{eq:var x+y is close};
the second line from Claim \ref{cla: const  | |}; the third from
\eqref{eq:var x is close} and \eqref{eq:var y is close}; the fourth
from the definition of $M_{XY}$ \eqref{eq:M_XY-def}; and the fifth
is true because $\left(1-2a\right)^{2}\geq1-4a$ and $V\leq\max\left\{ {\Var}\overline{X},{\Var}\overline{Y}\right\} /2$.)

\end{enumerate}
Combining the two cases, we have that
\[
{\Var}\left|\overline{X}+\overline{Y}+E\right|\geq M_{XY}\cdot\min\left\{ \frac{a}{K_{0}},\,\,\frac{1}{256}\frac{V}{V+E^{2}}-5a\right\} .
\]
Finally, we set $a=\frac{1}{2560}\cdot\frac{V}{V+E^{2}}$. Then,
\begin{eqnarray*}
{\Var}\left|\overline{X}+\overline{Y}+E\right| & \geq & \frac{1}{5120K_{0}}\cdot\frac{V\cdot M_{XY}}{V+E^{2}}
\end{eqnarray*}
and thus \eqref{eq: lemma statement} holds for $K_{1}=5120K_{0}\leq20480$.
\end{proof}

\subsection{\label{sub:main-thm}Proof of the main theorem}

Lemma \ref{lemma: Var|X+Y+c| > K * VarX * VarY} bounds the variance
of the absolute value of the sum of two independent variables, ${\Var}\left|X+Y\right|$,
in terms of the variance of each variable. In the following theorem
we generalize this claim to a sequence of $n$ independent variables.
\begin{theorem*}
\textup{(Theorem \ref{sum of sequence})} Let $\left(X_{i}\right)_{i=1}^{n}$
be a sequence of independent (not necessarily symmetric) random variables,
and let $E$ and $V$ be the expectation and variance of their sum,
respectively. Then for some universal constant $K_{2}\leq61440$ we
have that there exists $k\in\left[n\right]$ such that, 
\[
{\Var}\left|\sum_{i}X_{i}\right|\geq\frac{V\cdot{\Var}\sum_{i\neq k}X_{i}}{K_{2}\left(V+E^{2}\right)}.
\]
\end{theorem*}
\begin{proof}
In order to generalize Lemma \ref{lemma: Var|X+Y+c| > K * VarX * VarY}
to $n$ variables, consider the two possible cases:
\begin{enumerate}
\item If for every $i$, ${\Var}X_{i}\leq2V/3$, then we can partition $\left[n\right]$
into two sets $A,B$ such that 
\begin{equation}
\frac{V}{3}\leq\sum_{i\in A}{\Var}X_{i},\sum_{i\in B}{\Var}X_{i}\leq\frac{2V}{3}.\label{eq:13<V<23}
\end{equation}
(If ${\Var}X_{i}\leq V/3$, we can iteratively add variables to $A$
until \eqref{eq:13<V<23} is true; if $V/3<{\Var}X_{i}\leq2V/3$ for
some $i$, we can simply take $A=\left\{ i\right\} $.) Thus, substituting
$X=\sum_{i\in A}X_{i}$ and $Y=\sum_{i\in B}X_{i}$ in Lemma \ref{lemma: Var|X+Y+c| > K * VarX * VarY},
we have that for every $k$
\[
{\Var}\left|\sum_{i}X_{i}\right|={\Var}\left|\sum_{i\in A}X_{i}+\sum_{i\in B}X_{i}\right|\geq\frac{V\cdot\frac{V}{3}}{K_{1}\left(V+E^{2}\right)}\geq\frac{V\cdot\frac{\sum_{i\neq k}X_{i}}{3}}{K_{1}\left(V+E^{2}\right)}.
\]

\item Otherwise, if ${\Var}X_{k}>2V/3$, apply Lemma \ref{lemma: Var|X+Y+c| > K * VarX * VarY}
with $X=X_{k}$ and $Y=\sum_{i\neq k}X_{i}$ to get:

\[
{\Var}\left|\sum_{i}X_{i}\right|={\Var}\left|X_{k}+\sum_{i\neq k}X_{i}\right|\geq\frac{\frac{V}{3}\cdot\sum_{i\neq k}X_{i}}{K_{1}\left(V+E^{2}\right)}.
\]

\end{enumerate}
The theorem follows for $K_{2}=3K_{1}$.
\end{proof}

\subsection{Proof of the extension to FKN Theorem}

Corollary \ref{fkn-like}, the generalization of the FKN Theorem,
follows easily from Theorem \ref{sum of sequence}.
\begin{corollary*}
\textup{(Corollary \ref{fkn-like})} Let $f\colon\left\{ \pm1\right\} ^{m}\rightarrow\left\{ \pm1\right\} $
be a Boolean function, $\left(I_{j}\right)_{j=1}^{n}$ a partition
of $\left[m\right]$. Also, for each $I_{j}$ let $f_{j}$ be the
restriction of $f$ to the variables with indices in $I_{j}$. Suppose
that $f$ is concentrated on coefficients that do not cross the partition,
i.e.:
\[
\sum_{S\colon\exists j,\, S\subseteq I_{j}}\hat{f}{}^{2}\left(S\right)\geq1-\left(\epsilon\cdot{\Var}f\right).
\]
Then for some $k\in\left[n\right]$, $f$ is close to $f_{k}+\widehat{f}\left(\emptyset\right)$:
\[
\left\Vert f-f_{k}-\widehat{f}\left(\emptyset\right)\right\Vert _{2}^{2}\leq\left(K_{2}+2\right)\cdot\epsilon.
\]
\end{corollary*}
\begin{proof}
From the premise it follows that $f$ is $\epsilon\cdot{\Var}f$-close
to the sum of the $f_{j}$'s and the empty character:
\[
\left\Vert f-\sum_{j}f_{j}-\hat{f}\left(\emptyset\right)\right\Vert _{2}^{2}\leq\epsilon\cdot{\Var}f.
\]
Since $f$ is Boolean, this implies in particular that
\[
{\Var}\left|\sum_{j}f_{j}+\hat{f}\left(\emptyset\right)\right|\leq\epsilon\cdot{\Var}f.
\]
Thus by the main theorem, for some $k\in\left[n\right]$
\begin{eqnarray*}
\frac{{\Var}\sum_{j}f_{j}\cdot{\Var}\sum_{j\neq k}f_{j}}{K_{2}\cdot\left({\Var}\sum_{j}f_{j}+\hat{f}\left(\emptyset\right)^{2}\right)} & \leq & \epsilon\cdot{\Var}f.\\
{\Var}\sum_{j\neq k}f_{j} & \leq & K_{2}\cdot\epsilon\cdot{\Var}f\cdot\frac{{\Var}\sum_{j}f_{j}+\hat{f}\left(\emptyset\right)^{2}}{{\Var}\sum_{j}f_{j}}\\
{\Var}\sum_{j\neq k}f_{j} & \leq & K_{2}\cdot\epsilon.
\end{eqnarray*}
Rearranging and using ${\Var}\sum_{j}f_{j}+\hat{f}\left(\emptyset\right)^{2}\leq1$,
we have 
\[
{\Var}\sum_{j\neq k}f_{j}\leq K_{2}\cdot\epsilon\cdot\frac{{\Var}f}{{\Var}\sum_{j}f_{j}}.
\]
From the premise, we have ${\Var}\sum_{j}f_{j}\geq\left(1-\epsilon\right){\Var}f$,
and therefore
\begin{equation}
{\Var}\sum_{j\neq k}f_{j}\leq\frac{K_{2}}{1-\epsilon}\cdot\epsilon.\label{eq:k2-div-(1-epsilon)}
\end{equation}
Finally, we can assume without loss of generality that $\epsilon\leq1/\left(K_{2}+2\right)<1/\left(K_{2}+1\right)$,
and thus
\begin{gather*}
\frac{K_{2}}{1-\epsilon}<\frac{K_{2}}{1-\frac{1}{K_{2}+1}}=K_{2}+1.
\end{gather*}
Plugging back into \eqref{eq:k2-div-(1-epsilon)}, we get: 
\[
{\Var}\sum_{j\neq k}f_{j}<\left(K_{2}+1\right)\cdot\epsilon.
\]
Finally, 
\begin{eqnarray*}
\left\Vert f-f_{k}-\widehat{f}\left(\emptyset\right)\right\Vert _{2}^{2} & = & \sum_{S\colon S\nsubseteq I_{k}}\hat{f}{}^{2}\left(S\right)\\
 & \leq & {\Var}\sum_{j\neq k}f_{j}+\sum_{S\colon\forall j,\, S\nsubseteq I_{j}}\hat{f}{}^{2}\left(S\right)\\
 & \leq & \left(K_{2}+1\right)\cdot\epsilon+\epsilon{\Var}f\\
 & \leq & \left(K_{2}+2\right)\cdot\epsilon.
\end{eqnarray*}

\end{proof}

\section{\label{sec:Proofs-of-technical}Proofs of technical claims}

\subsection{\label{sub:Proving claim for every x,x,y,y}Expected squared distance:
proof of Claim \ref{cla:For-every x,x,y,y}}

We use case analysis to prove that adding any balanced $\overline{Y}$
with support of size two preserves (up to a factor of $\frac{1}{4}$)
the expected squared distance between any two possible evaluations
of $\left|\overline{X}+E\right|$. 
\begin{claim}
(Claim \ref{cla:For-every x,x,y,y}) For every two possible evaluations
$x_{1},x_{2}$ in the support of $\left(\overline{X}+E\right)$,
\[
{\E}_{\left(y_{1},y_{2}\right)\sim\overline{Y}\times\overline{Y}}\left(\left|x_{1}+y_{1}\right|-\left|x_{2}+y_{2}\right|\right)^{2}\geq\frac{1}{4}\left(\left|x_{1}\right|-\left|x_{2}\right|\right)^{2}.
\]
 \end{claim}
\begin{proof}
Denote 
\begin{eqnarray*}
p_{Y} & = & \Pr\left[\overline{Y}\geq0\right].
\end{eqnarray*}
Because $\overline{Y}$ is balanced, its two possible values must
be of the form $\left\{ \frac{d}{p_{Y}},\frac{-d}{1-p_{Y}}\right\} $,
for some $d\geq0$. Assume without loss of generality that $x_{1}\geq0$
and $\left|x_{1}\right|\geq\left|x_{2}\right|$. 

We divide our analysis to cases based on the value of $p_{Y}$ (see
also Figure \ref{fig:4_pt_argument}):

\begin{figure}[t]
\begin{centering}
\fbox{\parbox[t]{1\columnwidth}{%
\begin{center}
\includegraphics[width=1\textwidth]{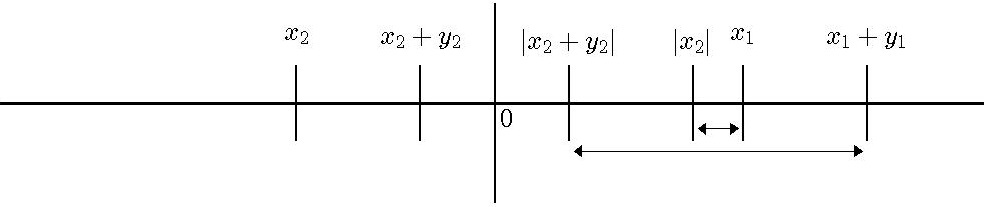}
\par\end{center}%
}}
\par\end{centering}

\begin{centering}
\fbox{\parbox[t]{1\columnwidth}{%
\begin{center}
\includegraphics[width=1\textwidth]{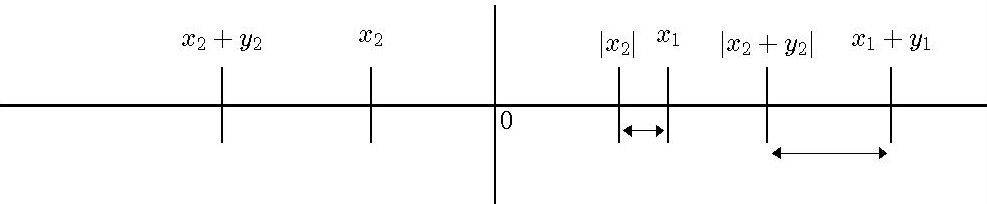}
\par\end{center}%
}}
\par\end{centering}

\protect\caption{\label{fig:4_pt_argument}Case analysis in the proof of Claim \ref{cla:For-every x,x,y,y}}

The top figure corresponds to the case where both $y_{1}$ and $y_{2}$
are non-negative, i.e. $y_{1}=y_{2}=\frac{d}{p_{Y}}\geq0$. Notice
that the distance between $\left|x_{1}+y_{1}\right|$ and $\left|x_{2}+y_{2}\right|$
is at least the distance between $\left|x_{1}\right|$ and $\left|x_{2}\right|$.
This case occurs with probability $p_{Y}^{2}$.

The bottom figure corresponds to the case where $p_{Y}<\frac{1}{2}$
and $y_{1}=\frac{d}{p_{Y}}>\frac{d}{1-p_{Y}}=\left|y_{2}\right|$.
Notice that in this case the distance also cannot decrease. In particular,
when $p_{Y}<\frac{1}{4}$, we have $y_{1}=\frac{d}{p_{Y}}>3\frac{d}{1-p_{Y}}=3\left|y_{2}\right|$,
and therefore the distance actually increases by a significant amount.
This case occurs with probability $p_{Y}\left(1-p_{Y}\right)$.
\end{figure}

\begin{enumerate}
\item If $p_{Y}\geq\frac{1}{2}$ then with probability at least $\frac{1}{4}$
both evaluations of $\overline{Y}$ are non-negative, in which case
the distance between $\left|x_{1}\right|$ and $\left|x_{2}\right|$
can only increase: 
\begin{eqnarray*}
{\E}_{y_{1},y_{2}}\left(\left|x_{1}+y_{1}\right|-\left|x_{2}+y_{2}\right|\right)^{2} & \geq & \Pr\left[y_{1},y_{2}\geq0\right]\cdot\left(x_{1}+\frac{d}{p_{Y}}-\left|x_{2}+\frac{d}{p_{Y}}\right|\right)^{2}\\
 & \geq & \Pr\left[y_{1},y_{2}\geq0\right]\cdot\left(x_{1}+\frac{d}{p_{Y}}-\left|x_{2}\right|-\frac{d}{p_{Y}}\right)^{2}\\
 & \geq & \frac{1}{4}\left(\left|x_{1}\right|-\left|x_{2}\right|\right)^{2}.
\end{eqnarray*}

\item If $\frac{1}{4}\leq p_{Y}<\frac{1}{2}$ then with probability at least
$\frac{1}{4}$, $y_{1}$ is non-negative; we also use $\frac{d}{p_{Y}}\geq\left|\frac{-d}{1-p_{Y}}\right|$
implies $y_{1}\geq\left|y_{2}\right|$: 
\begin{eqnarray*}
{\E}_{y_{1},y_{2}}\left(\left|x_{1}+y_{1}\right|-\left|x_{2}+y_{2}\right|\right)^{2} & \geq & \Pr\left[y_{1}\geq0\right]\cdot\left(\left|x_{1}\right|-\left|x_{2}\right|+y_{1}-\left|y_{2}\right|\right)^{2}\\
 & \geq & \Pr\left[y_{1}\geq0\right]\cdot\left(\left|x_{1}\right|-\left|x_{2}\right|\right)^{2}\\
 & \geq & \frac{1}{4}\left(\left|x_{1}\right|-\left|x_{2}\right|\right)^{2}.
\end{eqnarray*}

\item If $p_{Y}<\frac{1}{4}$ and $x_{1}\leq\frac{2d}{1-p_{Y}}$, we can
prove the claim by focusing on the case $y_{1}\geq0,y_{2}<0$: 
\begin{eqnarray}
{\E}_{y_{1},y_{2}}\left(\left|x_{1}+y_{1}\right|-\left|x_{2}+y_{2}\right|\right)^{2} & \geq & \Pr\left[y_{1}\geq0,y_{2}<0\right]\cdot\left(x_{1}-x_{2}+\frac{d}{p_{Y}}-\frac{d}{1-p_{Y}}\right)^{2}\label{eq:Pr=00005By1>0,y2<0=00005D}
\end{eqnarray}
Notice that $p_{Y}\leq\frac{1}{3}\left(1-p_{Y}\right)$ implies that
\[
\frac{d}{p_{Y}}-\frac{d}{1-p_{Y}}\geq\frac{2}{3}\cdot\frac{d}{p_{Y}}.
\]
Furthermore, since $p_{Y}\leq\frac{1}{4}$ and $1-p_{Y}\geq\frac{3}{4}$,
we have that 
\[
\frac{d}{p_{Y}}-\frac{d}{1-p_{Y}}\geq\frac{2}{3}\cdot\frac{\sqrt{p_{Y}}}{\frac{1}{2}}\cdot\frac{\left(\frac{3}{4}\right)^{3/2}}{\left(1-p_{Y}\right)^{3/2}}\cdot\frac{d}{p_{Y}}=\frac{\sqrt{\frac{3}{2}}}{\sqrt{p_{Y}\left(1-p_{Y}\right)}}\frac{d}{1-p_{Y}}\geq\frac{\sqrt{\frac{3}{8}}}{\sqrt{p_{Y}\left(1-p_{Y}\right)}}\left|x_{1}\right|.
\]
Plugging back into \eqref{eq:Pr=00005By1>0,y2<0=00005D} we have
\begin{eqnarray*}
{\E}_{y_{1},y_{2}}\left(\left|x_{1}+y_{1}\right|-\left|x_{2}+y_{2}\right|\right)^{2} & \geq & \Pr\left[y_{1}\geq0,y_{2}<0\right]\left(\frac{\sqrt{\frac{3}{8}}}{\sqrt{p_{Y}\left(1-p_{Y}\right)}}\left|x_{1}\right|\right)^{2}=\frac{3}{8}\left|x_{1}\right|^{2}.
\end{eqnarray*}

\item Else, if $p_{Y}<\frac{1}{4}$ and $x_{1}>\frac{2d}{1-p_{Y}}$, we
need to sum over the possible signs of $y_{1},y_{2}$, and use the
fact that $\frac{d}{p_{Y}}$ is much larger than $\left|\frac{-d}{1-p_{Y}}\right|$:
\begin{eqnarray*}
{\E}_{y_{1},y_{2}}\left(\left|x_{1}+y_{1}\right|-\left|x_{2}+y_{2}\right|\right)^{2} & \geq & \Pr\left[y_{1}\geq0,y_{2}\geq0\right]\cdot\left(\left|x_{1}\right|-\left|x_{2}\right|\right)^{2}+\\
 &  & \Pr\left[y_{1}\geq0,y_{2}<0\right]\cdot\underbrace{\left(\left|x_{1}\right|-\left|x_{2}\right|+\left(\frac{d}{p_{Y}}-\frac{d}{1-p_{Y}}\right)\right)^{2}}_{\left(a\right)}+\\
 &  & \Pr\left[y_{1}<0,y_{2}<0\right]\cdot\underbrace{\left(\left|x_{1}\right|-\left|x_{2}\right|-\frac{2d}{1-p_{Y}}\right)^{2}}_{\left(b\right)},
\end{eqnarray*}
where we used the condition $x_{1}>\frac{2d}{1-p_{Y}}$ in the third
line.

We have
\[
\left(a\right)\geq\left(\left|x_{1}\right|-\left|x_{2}\right|\right)^{2}+2\left(\frac{d}{p_{Y}}-\frac{d}{1-p_{Y}}\right)\left(\left|x_{1}\right|-\left|x_{2}\right|\right)
\]
and
\begin{gather*}
\left(b\right)\geq\frac{1}{4}\left(b\right)\geq\frac{1}{4}\left(\left(\left|x_{1}\right|-\left|x_{2}\right|\right)^{2}-2\frac{2d}{1-p_{Y}}\left(\left|x_{1}\right|-\left|x_{2}\right|\right)\right).
\end{gather*}
Therefore, 
\begin{align*}
{\E}_{y_{1},y_{2}}\left(\left|x_{1}+y_{1}\right|-\left|x_{2}+y_{2}\right|\right)^{2} & \geq  p_{Y}^{2}\left(\left|x_{1}\right|-\left|x_{2}\right|\right)^{2}+\\
 &   p_{Y}\left(1-p_{Y}\right)\left(\left(\left|x_{1}\right|-\left|x_{2}\right|\right)^{2}+2\left(\frac{d}{p_{Y}}-\frac{d}{1-p_{Y}}\right)\left(\left|x_{1}\right|-\left|x_{2}\right|\right)\right)+\\
 &   \frac{1}{4}\left(1-p_{Y}\right)^{2}\left(\left(\left|x_{1}\right|-\left|x_{2}\right|\right)^{2}-2\frac{2d}{1-p_{Y}}\left(\left|x_{1}\right|-\left|x_{2}\right|\right)\right)\\
 & \geq  \frac{1}{4}\left(\left|x_{1}\right|-\left|x_{2}\right|\right)^{2}+2d\left(\left(1-2\cdot\frac{1}{4}\right)\left(1-p_{Y}\right)-p_{Y}\right)\left(\left|x_{1}\right|-\left|x_{2}\right|\right)\\
 & \geq  \frac{1}{4}\left(\left|x_{1}\right|-\left|x_{2}\right|\right)^{2}.
\end{align*}

\end{enumerate}
\end{proof}

\subsection{\label{sub:Proof-of-analisys for const | |}Constant absolute value:
proof of Claim \ref{cla: const  | |}}

We use a brute-force case analysis to prove a relative lower bound
on the variance in absolute value of a sum of two variables with constant
absolute values:
\begin{claim}
(Claim \ref{cla: const  | |}) Let $\overline{X},\overline{Y}$ be
balanced random variables and let $X',Y'$ be the constant-absolute-value
approximations of $\overline{X},\overline{Y}$, respectively:
\begin{eqnarray*}
X' & = & \mbox{sign}\left(\overline{X}+E\right)\cdot{\E}\left|\overline{X}+E\right|\\
Y' & = & \mbox{sign}\left(\overline{Y}+E\right)\cdot{\E}\left|\overline{Y}+E\right|.
\end{eqnarray*}
Then the variance of the \emph{absolute value} of $X'+Y'-E$ is
bounded by:
\[
{\Var}\left|X'+Y'-E\right|\geq\frac{{\Var}\left(X'-E\right){\Var}\left(Y'-E\right)}{16\left({\Var}\overline{X}+E^{2}\right)}.
\]
\end{claim}
\begin{proof}
Denote $p_{X}=\Pr\left[\overline{X}+E\geq0\right]$ and $d_{X}={\E}\left|\overline{X}+E\right|$
(and analogously for $p_{Y},d_{Y}$). 

Observe that
\[
d_{X}\leq\sqrt{{\E}\left[\left(\overline{X}+E\right)^{2}\right]}=\sqrt{{\Var}\left(\overline{X}+E\right)+\left({\E}\left[\overline{X}+E\right]\right)^{2}}=\sqrt{{\Var}\overline{X}+E^{2}}.
\]
Thus we can bound $\left(1-p_{X}\right)p_{X}$ from below by:
\[
{\Var}X'=\left(2d_{X}\right)^{2}\Pr\left[X'\leq0\right]\Pr\left[X'>0\right]\leq4\left({\Var}\overline{X}+E^{2}\right)\left(1-p_{X}\right)p_{X}
\]
\begin{equation}
\left(1-p_{X}\right)p_{X}\geq\frac{{\Var}X'}{4\left({\Var}\overline{X}+E^{2}\right)}.\label{eq: (1-px) * px lower bound}
\end{equation}

Also, for $Y'$ we have
\[
{\Var}Y'=p_{Y}\cdot\left(1-p_{Y}\right)\left(2d_{Y}\right)^{2}.
\]
\begin{figure}[!t]
\begin{centering}
\fbox{\parbox[t]{1\columnwidth}{%
\begin{center}
\includegraphics[width=1\textwidth]{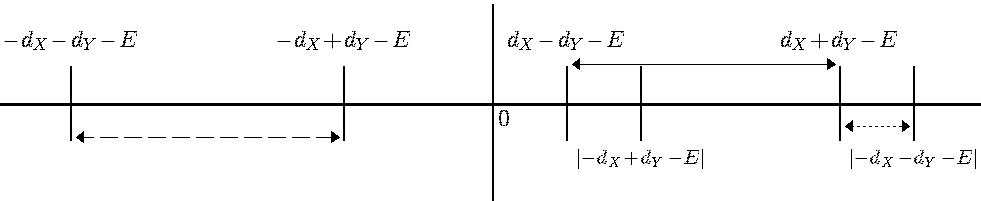}
\par\end{center}%
}}
\par\end{centering}

\begin{centering}
\fbox{\parbox[t]{1\columnwidth}{%
\begin{center}
\includegraphics[width=1\textwidth]{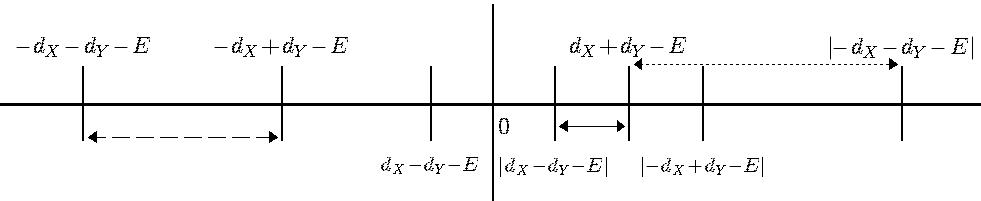}
\par\end{center}%
}}
\par\end{centering}

\protect\caption{\label{fig:const_abs_value}Case analysis in the proof of Claim \ref{cla: const  | |}}

\begin{itemize}
\item With probability $2\left(1-p_{X}\right)^{2}\cdot\left(1-p_{Y}\right)p_{Y}$
both $x$'s are negative, and $y$'s are of opposite signs. Notice
that since we assume $E\geq0$ and $d_{X}\geq d_{Y}$, the distance
between $\left|-d_{X}-d_{Y}-E\right|$ and $\left|-d_{X}+d_{Y}-E\right|$
is the same as the distance between $-d_{X}-d_{Y}-E$ and $-d_{X}+d_{Y}-E$
(marked by dashed line on both figures); it is therefore always $2d_{Y}$
. 
\item With probability $2p_{X}^{2}\cdot\left(1-p_{Y}\right)p_{Y}$ both
$x$'s are positive, and $y$'s are of opposite signs. Notice that
the distance between $\left|d_{X}+d_{Y}-E\right|$ and $\left|d_{X}-d_{Y}-E\right|$
(marked by the solid lines) is either $2d_{Y}$ (as in the top figure)
or $\left|2d_{X}-2E\right|$ when $d_{X}-d_{Y}-E\leq0$ (as in the
bottom figure). 
\item With probability $2\left(1-p_{X}\right)p_{X}\cdot\left(1-p_{Y}\right)p_{Y}$
the $x$'s and $y$'s are of correlated signs. Notice that the distance
between $\left|d_{X}+d_{Y}-E\right|$ and $\left|-d_{X}-d_{Y}-E\right|$
(marked by the dotted lines) is either $2E$ (as in both figures)
or $2d_{X}+2d_{Y}$ when $d_{X}+d_{Y}-E\leq0$ (not shown).\end{itemize}
\end{figure}

Assume without loss of generality that $d_{Y}<d_{X}$ and $E>0$.
Recall (Fact \ref{fac: (z1-z2)}) that we can write the variance in
terms of the expected squared distance between evaluations. Then,
summing over the different possible signs of $X'$ and $Y'$ we have
\begin{eqnarray*}
{\Var}\left|X'+Y'-E\right| & = & \frac{1}{2}{\E}_{x_{1},x_{2},y_{1},y_{2}\sim X'\times X'\times Y'\times Y'}\left(\left|x_{1}+y_{1}-E\right|-\left|x_{2}+y_{2}-E\right|\right)^{2}\\
 & \geq & \left(1-p_{X}\right)^{2}\cdot\left(1-p_{Y}\right)p_{Y}\cdot\left(\left|-d_{X}-d_{Y}-E\right|-\left|-d_{X}+d_{Y}-E\right|\right)^{2}+\\
 &  & p_{X}^{2}\cdot\left(1-p_{Y}\right)p_{Y}\cdot\left(\left|d_{X}-d_{Y}-E\right|-\left|d_{X}+d_{Y}-E\right|\right)^{2}+\\
 &  & \left(1-p_{X}\right)p_{X}\cdot\left(1-p_{Y}\right)p_{Y}\cdot\left(\left|-d_{X}-d_{Y}-E\right|-\left|d_{X}+d_{Y}-E\right|\right)^{2}\\
 & \geq & \left(\left(1-p_{X}\right)^{2}+p_{X}^{2}\right)\cdot\left(1-p_{Y}\right)p_{Y}\cdot\underbrace{\min\left\{ \left|2d_{X}-2E\right|,2d_{Y}\right\} ^{2}}_{\left(a\right)}+\\
 &  & \left(1-p_{X}\right)p_{X}\cdot p_{Y}\left(1-p_{Y}\right)\cdot\underbrace{\min\left\{ 2E,2d_{X}+2d_{Y}\right\} ^{2}}_{\left(b\right)}
\end{eqnarray*}
(Where the first inequality follows by taking the expectation over
the different possible signs of $X'$, $Y'$ (see also Figure \ref{fig:const_abs_value});
the second follows by taking the minimum over the possible signs of
the quantities in absolute values;)

We next claim that 
\begin{equation}
\left(a\right)+\left(b\right)\geq d_{Y}^{2}\label{eq:(a)+(b)>dY}
\end{equation}
If $2E\geq d_{Y}$, \eqref{eq:(a)+(b)>dY} is immediate. If $2E<d_{Y}$,
then \eqref{eq:(a)+(b)>dY} follows because $2d_{X}-2E\ge d_{x}\ge d_{Y}$. 

Therefore, 
\begin{eqnarray*}
{\Var}\left|X'+Y'-E\right| & \geq & \left(1-p_{X}\right)p_{X}\cdot p_{Y}\left(1-p_{Y}\right)\cdot d_{Y}^{2}\\
 & \geq & \frac{{\Var}X'}{4\left({\Var}\overline{X}+E^{2}\right)}\cdot p_{Y}\left(1-p_{Y}\right)\cdot\frac{\left(2d_{Y}\right)^{2}}{4}\\
 & = & \frac{{\Var}X'}{16\left({\Var}\overline{X}+E^{2}\right)}\cdot{\Var}Y'\\
 & = & \frac{{\Var}\left(X'-E\right){\Var}\left(Y'-E\right)}{16\left({\Var}\overline{X}+E^{2}\right)}.
\end{eqnarray*}
(Where the first inequality follows from \eqref{eq:(a)+(b)>dY} and
$\left(1-p_{X}\right)^{2}+p_{X}^{2}>\left(1-p_{X}\right)p_{X}$; and
the second inequality follows by \eqref{eq: (1-px) * px lower bound}.)
\end{proof}

\section{\label{sec:Tightness-of-results}Tightness of results}

\subsection{Tightness of \label{sub:Tightness-of-the-main}the main result}

The premise of main result, corollary \ref{fkn-like}, requires $f$
to be $\left(\epsilon\cdot{\Var}f\right)$-close to a sum of independent
functions. One may hope to avoid this factor of ${\Var}f$ and achieve
a constant ratio between $\epsilon$ in the premise and $\left(K\cdot\epsilon\right)$
in the conclusion, as in the FKN Theorem. However, we show that the
dependence on ${\Var}f$ is necessary.
\begin{lemma*}
\textup{(Lemma \ref{tightness of main result-1})} Corollary \ref{fkn-like}
is tight up to a constant factor. In particular, the factor ${\Var}f$
is necessary. 

More precisely, there exists a sequence of functions $f^{\left(m\right)}\colon\left\{ \pm1\right\} ^{2m}\rightarrow\left\{ \pm1\right\} $
and partitions $\left(I_{1}^{\left(m\right)},I_{2}^{\left(m\right)}\right)$
such that the restrictions $\left(f_{1}^{\left(m\right)},f_{2}^{\left(m\right)}\right)$
of $f^{\left(m\right)}$ to variables in $I_{j}^{\left(m\right)}$
satisfy 
\[
\sum_{S\colon\exists j,\, S\subseteq I_{j}^{\left(m\right)}}\hat{f}{}^{2}\left(S\right)=1-O\left(2^{-m}\cdot{\Var}f\right).
\]
but for every $j\in\left\{ 1,2\right\} $
\[
\left\Vert f^{\left(m\right)}-f_{j}^{\left(m\right)}-\widehat{f^{\left(m\right)}}\left(\emptyset\right)\right\Vert _{2}^{2}=\Theta\left(2^{-m}\right)=\omega\left(2^{-m}\cdot{\Var}f\right).
\]
\end{lemma*}
\begin{proof}
By example. Let 
\begin{eqnarray*}
X & = & \bigwedge_{i=1}^{m}x_{i}\\
Y & = & \bigwedge_{i=1}^{m}y_{i}\\
f & = & X\vee Y,
\end{eqnarray*}
where we think of $-1$ as ``true'' and $1$ as ``false''.

The variance of $f$ is $\Theta\left(2^{-m}\right)$:
\begin{gather*}
{\Var}f=4\Pr\left[f=1\right]\Pr\left[f=-1\right]=4\left(1-\Theta\left(2^{-m}\right)\right)\Theta\left(2^{-m}\right)=\Theta\left(2^{-m}\right).
\end{gather*}
Also, $f$ is $O\left(2^{-2m}\right)$-close to a sum of independent
functions:
\begin{eqnarray*}
f & = & \frac{X+Y+X\cdot Y-1}{2}\\
\left\Vert f-\left(X+Y-1\right)\right\Vert _{2}^{2} & = & \left\Vert \frac{\left(X-1\right)\left(Y-1\right)}{2}\right\Vert _{2}^{2}=4\cdot2^{-2m}.
\end{eqnarray*}
Yet, $f$ is $\Omega\left(2^{-m}\right)$-far from any function that
depends on either only the $x_{i}$'s or only the $y_{i}$'s.
\end{proof}

\subsection{Tightness of\label{sub:Tightness-of-the-aux lemma} Lemma \ref{lemma: |X+Y| > K*|X+E=00005BY=00005D|}}

Lemma \ref{lemma: |X+Y| > K*|X+E=00005BY=00005D|} compares the variance
of the absolute value of a sum of independent variables, to the variance
of the absolute value of each variable. Since both sides of the inequality
consider absolute values, it may seem as if we should only be increasing
the variation on the left side by summing independent variables. In
particular, one may hope that the inequality should hold trivially,
with $K_{0}=1$. We show that this is not the case.
\begin{claim}
(Claim \ref{cla: tightness of aux lemma}) A non-trivial constant
is necessary for Lemma \ref{lemma: |X+Y| > K*|X+E=00005BY=00005D|}.
More precisely, there exist two independent \emph{balanced} random
variables $\overline{X},\overline{Y}$, such that the following inequality
does not hold for any value $K_{0}<4/3$:
\[
{\Var}\left|\overline{X}+\overline{Y}\right|\geq\frac{\max\left\{ {\Var}\left|\overline{X}\right|,{\Var}\left|\overline{Y}\right|\right\} }{K_{0}}
\]

\end{claim}
(In particular, it is interesting to note that $K_{0}>1$.)
\begin{proof}
By example. Let
\begin{eqnarray*}
\Pr\left[X=0\right] & = & \frac{1}{2}\\
\Pr\left[X=\pm2\right] & = & \frac{1}{4}\\
\Pr\left[Y=\pm1\right] & = & \frac{1}{2}.
\end{eqnarray*}
Then we have that
\[
{\E}X={\E}Y=0
\]
\begin{eqnarray*}
\Pr\left[\left|X+Y\right|=1\right] & = & \frac{3}{4}\\
\Pr\left[\left|X+Y\right|=3\right] & = & \frac{1}{4},
\end{eqnarray*}
and therefore
\[
{\Var}\left|X+Y\right|=\frac{3}{4}=\frac{3}{4}{\Var}\left|X\right|.
\]

\end{proof}

\section{\label{sec:Conjectures-and-extensions}Conjectures and extensions}

While the dependence on the variance in Corollary \ref{fkn-like}
is tight, it seems counter-intuitive. We believe that it is possible
to come up with a structural characterization instead. 

Observe that the function used for the counter example in Lemma \ref{tightness of main result-1}
is essentially the (non-balanced) \emph{tribes} function, i.e. $OR$
of two $AND$'s. All the extreme examples we have discovered so far
have a similar structure of an independent Boolean function on each
subset of the variables (e.g. $AND$ on a subset of the variables),
and then a ``central'' Boolean function that takes as inputs the
outputs of the independent functions (e.g. $OR$ of all the $AND$'s). 

We conjecture that such a \emph{composition} of Boolean functions
is essentially the only way to construct counterexamples to the ``naive
extension'' of the FKN Theorem. In other words, if a Boolean function
is close to linear with respect to a partition of the variables, then
it is close to the application of a central Boolean function $g$
on the outputs of independent Boolean functions $g_{j}$'s, one over
each subset $I_{j}$. Formally,
\begin{conjecture}
Let $f\colon\left\{ \pm1\right\} ^{m}\rightarrow\left\{ \pm1\right\} $
be a Boolean function, $\left(I_{j}\right)_{j=1}^{n}$ a partition
of $\left[m\right]$. Suppose that $f$ is concentrated on coefficients
that do not cross the partition, i.e.:
\[
\sum_{S\colon\exists j,\, S\subseteq I_{j}}\hat{f}{}^{2}\left(S\right)\geq1-\epsilon.
\]
Then there exist a ``central'' Boolean function $g\colon\left\{ \pm1\right\} ^{n}\rightarrow\left\{ \pm1\right\} $
and a Boolean function on each subset $h_{j}\colon\left\{ \pm1\right\} ^{\left|I_{j}\right|}\rightarrow\left\{ \pm1\right\} $
such that the composition of $g$ with the $h_{j}$'s is a good approximation
of $f$. I.e. for some universal constant $K$, 
\[
\left\Vert f\left(X\right)-g\left(h_{1}\left(\left(x_{i}\right)_{i\in I_{1}}\right),h_{2}\left(\left(x_{i}\right)_{i\in I_{2}}\right),\dots,h_{n}\left(\left(x_{i}\right)_{i\in I_{n}}\right)\right)\right\Vert _{2}^{2}\leq K\cdot\epsilon.
\]

\end{conjecture}
Intuitively, this conjecture claims that the central function only
needs to \emph{know} one bit of information on each subset in order
to approximate $f$. 

We believe that such a conjecture could have useful applications because
one can often deduce properties of the composition of independent
functions $f=g\left(h\left(x_{I_{1}}\right),h\left(x_{I_{2}}\right),\dots,h\left(x_{I_{n}}\right)\right)$
from the properties of the composed functions $g$ and $h$. For example
if $f$, $g$, and $h$ are as above, then the total influence of
$f$ is the product of the total influences of $g$ and $h$.

In fact, we believe that an even stronger claim holds. It seems that
for all the Boolean functions that are almost linear with respect
to a partition of the variables, the ``central'' function $g$ is
either an $OR$ or an $AND$ of some of the functions on the subsets
$h_{j}$. Formally,
\begin{conjecture}
\textup{(Stronger variant)} Let $f\colon\left\{ \pm1\right\} ^{m}\rightarrow\left\{ \pm1\right\} $
be a Boolean function, $\left(I_{j}\right)_{j=1}^{n}$ a partition
of $\left[m\right]$. Suppose that $f$ is concentrated on coefficients
that do not cross the partition, i.e.:
\[
\sum_{S\colon\exists j,\, S\subseteq I_{j}}\hat{f}{}^{2}\left(S\right)\geq1-\epsilon.
\]
Then there exist Boolean functions $h_{j}\colon\left\{ \pm1\right\} ^{\left|I_{j}\right|}\rightarrow\left\{ \pm1\right\} $
for each $j\in\left[n\right]$ such that either the $OR$ or the $AND$
of those $h_{j}$'s is a good approximation of $f$. I.e. for some
universal constant $K$, 
\begin{gather*}
\left\Vert f\left(X\right)-OR_{j\in\left[n\right]}\left(h_{j}\left(\left(x_{i}\right)_{i\in I_{j}}\right)\right)\right\Vert _{2}^{2}\leq K\cdot\epsilon\\
\mbox{-or-}\\
\left\Vert f\left(X\right)-AND_{j\in\left[n\right]}\left(h_{j}\left(\left(x_{i}\right)_{i\in I_{j}}\right)\right)\right\Vert _{2}^{2}\leq K\cdot\epsilon.
\end{gather*}
\end{conjecture}
\begin{acknowledgement*}
An earlier version of this paper appeared before as the Masters thesis
of the first author. 

We are grateful to anonymous referees for pointing out errors in a
previous draft, as well as many helpful comments and suggestions.
\end{acknowledgement*}
\bibliographystyle{plain}
\bibliography{paper}

\appendix

\section{Proofs of preliminary facts}

Below, we bring missing proofs of facts from Section \ref{sec:Preliminaries}.
All these proofs can be found elsewhere (e.g. \cite{ODonnell14-book}),
and are brought here only for completeness.
\begin{fact}
(Fact \ref{fac:l2_norm-ft})
\[
\left\Vert f-g\right\Vert _{2}^{2}=\sum\left(\widehat{f}\left(S\right)-\widehat{g}\left(S\right)\right)^{2}.
\]
\end{fact}
\begin{proof}
\[
\left\Vert f-g\right\Vert _{2}^{2}=\sum\left(\widehat{f-g}\left(S\right)\right)^{2}=\sum\left(\widehat{f}\left(S\right)-\widehat{g}\left(S\right)\right)^{2}.
\]
\end{proof}
\begin{fact}
(Fact \ref{fac:var-ft})
\[
{\Var}f=\sum_{S\neq\emptyset}\widehat{f}\left(S\right)^{2}.
\]
\end{fact}
\begin{proof}
\begin{eqnarray*}
{\Var}f & = & {\E}\left[f^{2}\right]-\left({\E}f\right)^{2}\\
 & = & {\E}\left[\left(\sum_{S}\widehat{f}\left(S\right)\chi_{S}\right)^{2}\right]-\left({\E}\sum_{S}\widehat{f}\left(S\right)\chi_{S}\right)^{2}\\
 & = & \left(\sum_{S,T}\widehat{f}\left(S\right)\widehat{f}\left(T\right){\E}\chi_{S}\chi_{T}\right)-\left(\sum_{S}\widehat{f}\left(S\right){\E}\chi_{S}\right)^{2}\\
 & = & \left(\sum_{S}\widehat{f}\left(S\right)^{2}\right)-\left(\widehat{f}\left(\emptyset\right)\right)^{2}.
\end{eqnarray*}
\end{proof}
\begin{fact}
(Fact \ref{fac: (z1-z2)}) For any random variable X, \textup{
\[
{\Var}X=\frac{1}{2}\cdot{\E}_{x_{1},x_{2}\sim X\times X}\left(x_{1}-x_{2}\right)^{2}.
\]
}\end{fact}
\begin{proof}
\[
{\E}_{x_{1},x_{2}\sim X\times X}\left(x_{1}-x_{2}\right)^{2}={\E}x_{1}^{2}+{\E}x_{2}^{2}-2{\E}x_{1}x_{2}=2\left({\E}X^{2}-\left({\E}X\right)^{2}\right)=2{\Var}X.
\]
\end{proof}
\begin{fact}
(Fact \ref{fac:var-l2_norm})
\[
{\Var}X=\left\Vert X-{\E}X\right\Vert _{2}^{2}.
\]
\end{fact}
\begin{proof}
\[
{\E}\left[X^{2}\right]-\left({\E}X\right)^{2}={\E}\left[X^{2}\right]-2{\E}\left[X{\E}X\right]+{\E}\left[\left({\E}X\right)^{2}\right]={\E}\left[\left(X-{\E}X\right)^{2}\right].
\]
\end{proof}
\begin{fact}
(Fact \ref{fac:var-min_l2_norm})
\[
{\Var}X=\min_{E\in\mathbb{R}}\left\Vert X-E\right\Vert _{2}^{2}.
\]
\end{fact}
\begin{proof}
Differentiate twice with respect to $E$:
\begin{eqnarray*}
\frac{d}{dE}\left\Vert X-E\right\Vert _{2}^{2} & = & 2\left({\E}X-E\right)\\
\frac{d^{2}}{dE^{2}}\left\Vert X-E\right\Vert _{2}^{2} & = & -2.
\end{eqnarray*}
\end{proof}

\end{document}